\documentclass[12pt]{amsart}
\usepackage{mathrsfs}
\newtheorem{theorem}{Theorem}
\newtheorem{lemma}{Lemma}
\newtheorem{corollary}{Corollary}
\newtheorem{proposition}{Proposition}

\begin{document}
\title[Isoparametric hypersurfaces]{Isoparametric hypersurfaces with four
principal curvatures, II}
\author{Quo-Shin Chi}
\thanks{The author was partially supported by NSF Grant No. DMS-0103838}
\address{Department of Mathematics, Washington University, St. Louis, MO 63130}
\email{chi@math.wustl.edu}
\date{}

\begin{abstract}
In this sequel to~\cite{CCJ}, employing more commutative algebra than that explored
in~\cite{CCJ}, we show that an isoparametric hypersurface with four principal curvatures
and multiplicities
$(3,4)$ in $S^{15}$ is one constructed by Ozeki-Takeuchi~\cite[I]{OT}
and Ferus-Karcher-M\"{u}nzner~\cite{FKM}, referred to collectively as of OT-FKM type.

In fact, this new approach also gives a considerably simpler, both structurally
and technically, proof~\cite{CCJ} that an isoparametric
hypersurface with four
principal curvatures in spheres with the multiplicity
constraint $m_2\geq 2m_1-1$ is of OT-FKM type, which left unsettled exactly
the four anomalous multiplicity pairs $(4,5),(3,4),(7,8)$ and $(6,9)$,
where the last three are closely tied, respectively, with the quaternion
algebra, the octonion algebra and the
complexified octonion algebra, whereas
the first stands alone by itself in that it cannot be of OT-FKM type.

A byproduct of this new approach is that we see that Condition B, introduced by
Ozeki and Takeuchi~\cite[I]{OT} in their construction of inhomogeneous isoparametric hypersurfaces,
naturally arises.

The cases for the multiplicity pairs $(4,5),(6,9)$ and $(7,8)$ remain open now.
\end{abstract}


\maketitle
\footnote{The paper is a slightly revised version of arXiv:1002.1345.}

\section{Introduction} An isoparametric hypersurface in a space form is a
complete hypersurface whose
principal curvatures and their multiplicities are fixed constants.
The long history of the
study of isoparametric hypersurfaces
dates back to 1918 when isoparametric surfaces in Euclidean 3-space arose in the study
of geometric optics~\cite{La},~\cite{So},~\cite{Se1}; in contrast, their latest application to integrable systems
came in as late as in 1995~\cite{Fe}, to the author's knowledge.
The classification problem of isoparametric hypersurfaces started when
Segre~\cite{Se2},
for ambient dimension 3, and Levi-Civita~\cite{Le}, for arbitrary ambient dimension $n$,
classified such hypersurfaces
in Euclidean space; they are none other than the cylinders $S^k\times{\mathbb R}^{n-k-1}$. 
Cartan then took up the task and quickly settled the hyperbolic case~\cite{Car1}; again
the hyperbolic cylinders $S^k\times H^{n-k-1}$ are the only ones. The spherical case
amazed Cartan as such hypersurfaces
displayed remarkably deep properties. In fact, he classified the case when
the number $g$ of principal curvatures is $\leq 3$. Clifford tori
$S^k(r)\times S^{n-k-1}(s)\subset S^n,r^2+s^2=1,$ constitute
the case when $g=2$. For $g=3$, he showed that such hypersurfaces
are tubes of constant radii around
the Veronese embedding of the projective plane ${\mathbb F}P^2$
in $S^{3m+1}$, where $m=1,2,4$ or $8$ is the dimension of the standard 
normed algebra ${\mathbb F}={\mathbb R},{\mathbb C},{\mathbb H},$ or the Cayley
algebra ${\mathbb O}$, respectively,
~\cite{Car2},~\cite{Car3}. In particular,
this gives a very geometric
description of the Cayley  plane that appears in the classification of the
rank-one symmetric spaces of compact type that Cartan had classified
earlier using Lie group theory he developed. Therefore, all these isoparametric hypersurfaces
are homogeneous.

To make the long story short, leaving the beautiful subsequent
development of the classification problem
to the introductory section in~\cite{CCJ},
let us simply remark that there holds~\cite{Mu} that $g$ can only be 1, 2, 3, 4, 6; moreover,
the principal values have at most two multiplicities. In fact,
if we list these principal values as $\lambda_1<\cdots<\lambda_g,$ and their
associated multiplicities as $m_1,\cdots,m_g$, respectively, then $m_i=m_{i+2}$,
where the subscripts are modulo $g$. In particular,
$m_1=m_2$ (actually, $=1,2,4,$ or $8$ associated with the standard normed
algebras) when $g=3$. We will denote these
two multiplicities by $m_1$ and $m_2$, with the understanding that
$m_1\leq m_2$.
To each isoparametric hypersurface in the sphere, there is indeed a 1-parameter
family of isoparametric hypersurfaces in the sphere that degenerates to two submanifolds
of the sphere of codimensions $m_1+1$ and $m_2+1$, called the focal manifolds
of the isoparametric hypersurface and denoted by $M_{+}$ and $M_{-}$,
respectively. 
The principal values of the focal manifold of codimension $m_1+1$ are
$0,1,-1$ with multiplicities $m_1,m_2,m_2$, respectively, with $m_1$ and $m_2$
interchanged for the other focal manifold.

It is known~\cite{A} that $m_1=m_2=1$ or $2$ in the case $g=6$. The case $m_1=m_2=1$ was settled
by Dorfmeister and Neher~\cite{DN} (see also Miyaoka~\cite{Mi1}). The other case has recently been settled by
Miyaoka~\cite{Mi2}. Such isoparametric hypersurfaces are again homogeneous.

Cartan~\cite{Car4} found two homogeneous examples of isoparametric hypersurfaces with $g=4$
and $(m_1,m_2)=(2,2)$ or $(4,5)$. He indicated without proof that his example with
multiplicity pair $(2,2)$ is the only isoparametric hypersurface with the given
$g$ and multiplicities; an outline of a proof was given by Ozeki and Takeuchi~\cite[II, pp 53-54]{OT}.

In contrast, inhomogeneous examples do appear in the case $g=4$.
Ozeiki and Takeuchi~\cite[I, p 541, p 549]{OT}
introduced the notions of Conditions A and B to first construct two families of them~\cite[I]{OT}, followed by Ferus, Karcher and
M\"{u}nzner's generalization to infinite families using Clifford
representations~\cite{FKM}, where $m_1$ and $m_2$ are explicitly given in terms of
the dimensions of irreducible Clifford modules. We refer to these examples collectively as of
OT-FKM type. These examples, together with the two aforementioned homogeneous examples of
Cartan, not of OT-FKM type, have constituted all known
isoparametric hypersurfaces with four principal curvatures in spheres. In fact,
using homotopy theory, Stolz proved~\cite{St} that the multiplicity pair of an isoparametric
hypersurface with $g=4$ in the sphere is either $(2,2),(4,5),$ or one of those of isoparametric
hypersurfaces of OT-FKM type.
In~\cite{CCJ}, we classified that, except for possibly
the four cases with $(m_1,m_2)=(4,5),(3,4),(6,9)$ or $(7,8)$, all
other isoparametric hypersurfaces with $(m_1,m_2)\neq (2,2)$ are necessarily
of OT-FKM type.

It is worth pointing out that isoparametric hypersurfaces of OT-FKM type with multiplicity
pairs $(m_1,m_2)=(3,4),(7,8)$ and $(6,9)$ are anomalous in the sense that they are the only ones
for which the two symmetric Clifford algebras $C'_{m_1+1}$ and $C'_{m_2+1}$ can act
on ${\mathbb R}^{2m_1+2m_2+2}$ to produce incongruent isoparametric hypersurfaces
in $S^{2m_1+2m_2+1}$~\cite{FKM}. Their isoparametric structures are tied, respectively, with
${\mathbb H},{\mathbb O}$ and ${\mathbb O}({\mathbb C})$~\cite{OT}. On the other hand,
an isoparametric hypersurface with $g=4$ and multiplicities $(m_1,m_2)=(4,5)$ stands alone by itself
as it cannot be of OT-FKM type. 

The classification theorem in~\cite{CCJ} (see also~\cite{Ch}), where it is assumed
$m_2\geq 2m_1-1$, eventually comes down to an 
estimate on the dimension of certain singular varieties associated with
the complexified second fundamental form of the focal manifold of
codimension $m_1+1$, restricted to the direct sum of the two eigenspaces
of the principal values $1,-1$. This restriction to a subspace of a typical
tangent space of the focal manifold complicates the dimension estimate,
since the restriction does not preserve the constant rank that the second fundamental
matrix of a focal manifold enjoys. Moreover, the tracking of the interplay
between real and complex varieties in the proof brings in additional technicality.

The primary goal in this sequel to~\cite{CCJ} is to prove that an isoparametric hypersurface
with $g=4$ and $(m_1,m_2)=(3,4)$ is one of OT-FKM type as well,
which leaves only three multiplicity pairs $(4,5),(6,9)$ and $(7,8)$ unsettled. In fact,
this grew out of an attempt to look at the classification theorem in~\cite{CCJ}
from a different angle.

Indeed, employing more commutative algebra than that explored
in~\cite{CCJ}, we first present a considerably simpler proof of the
classification theorem in~\cite{CCJ} by investigating the (complexified)
second fundamental form itself without
further restriction. An advantage of this approach is that we see 
that Condition B of
Ozeki and Takeuchi, which all isoparametric hypersurfaces of
OT-FKM type enjoy as a matter of fact, naturally arises.

The key ingredient that this new proof of the classification theorem
falls on is the satisfying result (Proposition~\ref{prop2})
that states that for $m_1<m_2$, if the components $p_0,p_1,\cdots,p_{m_1}$ of
the second fundamental form of $M_{+}$ form a regular sequence
in the ring of polynomials in $m_1+2m_2$
variables (over the complex numbers),
then the isoparametric hypersurface is of OT-FKM type. The result sheds illuminating light
on why the remaining multiplicity pairs $(3,4),(6,9)$ and $(7,8)$ are anomalous.
It is because in these cases $p_0,\cdots,p_{m_1}$ no longer form a regular sequence
in general, due to the aforementioned fact that incongruent
isoparametric hypersurfaces of OT-FKM type do occur in $S^{2m_1+2m_2+1}$; the zero
locus cut out by a nonregular sequence $p_0,\cdots,p_{m_1}$ can be
wildly untamed, even in the complex category.
In fact, Proposition~\ref{prop2} follows from one of
the ten identities defining an isoparametric
hypersurface~\cite[I, p 530]{OT}. Namely, 
$$
\sum_{a=0}^{m_1}p_aq_a=0,
$$
where $q_0,q_1,\cdots,q_{m_1}$ are the components of the third fundamental
form of $M_{+}$. That
$p_0,\cdots,p_{m_1}$ form a regular sequence
ensures that
$$
q_a=\sum_{b=0}^{m_1} r_{ab}p_b,
$$
where $0\leq a\leq m_1$, for some linear homogeneous polynomials $r_{ab}$
satisfying
$$
r_{ab}=-r_{ba}
$$
for $0\leq a,b\leq m_1$. This is exactly Condition B of Ozeki and Takeuchi, from which there readily
follow the three equations (8.1) through (8.3) of Proposition 19 of~\cite{CCJ}.
The mild condition $m_1<m_2$ then warrants that (8.4) of~\cite{CCJ} also holds true; the
isoparametric hypersurface is then of OT-FKM type~\cite{CCJ},~\cite{Ch}.
That $p_0,\cdots,p_{m_1}$ indeed
form a regular sequence,
when $m_2\geq 2m_1-1$,
is a consequence of a rather straightforward dimension estimate, facilitated by the constant
rank of the shape operators of $M_{+}$, on where the Jacobian
matrix of these polynomials fail to be of rank $m_1+1$ on the variety cut
out by them. 

What is more important is that in fact the new approach in this paper provides
us with a proof that an isoparametric hypersurface with $g=4$
and $(m_1,m_2)=(3,4)$ in $S^{15}$ is of OT-FKM type,
by showing that there exist on $M_{+}$ points of Condition A of
Ozeki and Takeuchi. The conclusion then follows from the result of
Dorfmeister and Neher~\cite{DN1} (see also~\cite{Chi}) that states 
that Condition A alone implies that the isoparametric hypersurface
is of OT-FKM type.

The salient feature of this approach would appear
applicable also to the case $(m_1,m_2)=(7,8)$. We will comment on this
in the concluding remarks.

\section{The background commutative algebra} Recall~\cite[p 152]{Ku} that a {\em regular sequence}\, in a
commutative ring $R$ with identity
is a sequence $a_1,\cdots,a_k$ in $R$ such that the ideal $(a_1,\cdots,a_k)$ is
not $R$, and moreover, $a_1$ is not a zero divisor in $R$ and $a_{i+1}$ is not
a zero divisor in the quotient ring $R/(a_1,\cdots,a_i)$ for $1\leq i\leq k-1$.


There is a powerful property in commutative algebra that dictates
the algebraic independence of a regular sequence~\cite[Proposition 5.10, p 152]{Ku},
which is specialized to fit our purpose in the following proposition.

\begin{proposition}\label{prop1}
Let $a_1,\cdots,a_k$ be a regular sequence in $R$.
Then for any homogeneous polynomial $F(t_1,\cdots,t_k)$ in $k$ variables over $R$
with $F(a_1,\cdots,a_k)=0$, there always holds that all the coefficients of $F$
belong to $(a_1,\cdots,a_k)$.
\end{proposition}

For the convenience of the reader, let us recall a crucial inductive procedure
in~\cite[Proposition 39, p 57]{CCJ}
to generate regular sequences in a polynomial ring.

\begin{proposition}\label{prop1.5} Over the complex numbers, if
$p_1,\cdots,p_k,k\geq 2,$ are linearly independent homogeneous
polynomials of equal degree $\geq 1$ in a polynomial ring $P$ such that the ideal
$(p_1,\cdots,p_{k-1})$ is prime and that $p_1,\cdots,p_{k-1}$ form a regular sequence
in $P$, then $p_1,\cdots,p_k$ form a regular sequence in $P$.
\end{proposition}

To warrant the primeness of an ideal $(p_1,\cdots,p_s)$ in a polynomial ring,
the following~\cite[Proposition 43, p 59]{CCJ} is essential.

\begin{proposition}\label{prop1.6} Over the complex numbers, let
$p_1,\cdots,p_s$
be a regular sequence of homogeneous polynomials in a polynomial ring, let
$V$ be the variety defined by $p_1=\cdots=p_s=0$, and let
$J$ be the subvariety of $V$ where the rank of the Jacobian matrix of
$p_1,\cdots,p_s$
is $<s$. If $\dim(J)\leq \dim(V)-2$, then the ideal $(p_1,\cdots,p_s)$ is prime.
\end{proposition}
Note that the homogeneity of $p_1,\cdots,p_s$ is to guarantee that the variety $V$
is connected.

\begin{corollary}\label{cor1}
Over the complex numbers, let $p_1,\cdots,p_k,k\geq 2,$ be linearly
independent homogeneous polynomials of equal degree $\geq 1$
in a polynomial ring. For $i\leq k$, let $V_i$ be the variety defined by 
$p_1=\cdots=p_i=0,$ and let $J_i$ be the subvariety of $V_i$ where the rank of
Jacobian matrix of $p_1,\cdots,p_i$ is $<i$. If $\dim(J_i)\leq\dim(V_i)-2$
for $1\leq i\leq k-1$, then $p_1,\cdots,p_k$ form a regular sequence in the polynomial ring.
\end{corollary}

For a proof, note that $p_1$ clearly forms a regular sequence, so that Proposition~\ref{prop1.6}
implies that $(p_1)$ is a prime ideal; so the corollary is true when $i=1$.
Repeated applications of Propositions~\ref{prop1.5} and~\ref{prop1.6} then
deduce that the
ideal $(p_1,\cdots,p_i)$ are prime and that
$p_1,\cdots,p_i$ form a regular sequence for $1\leq i\leq k-1$. From this Proposition~\ref{prop1.5}
gives that $p_1,\cdots,p_k$ form a regular sequence in the polynomial ring.

\section{The game plan}
We will follow closely the notations in~\cite{CCJ} for ease of exposition.
Let
$n_0,n_1,\cdots,n_{m_1}$ be an orthonormal basis of a typical normal space to
the focal manifold $M_{+}$ of codimension $m_1+1$. Let
$p_0,p_1,\cdots,p_{m_1}$ be the associated symmetric quadratic forms associated with
the second fundamentl form $S$ of $M_{+}$. That is, 
$$
p_a(X):=<S(X,X),n_a>/2,
$$
for $0\leq a\leq m_1$.
Let $q_0,q_1,\cdots,q_{m_1}$ be the associated symmetric cubic forms of the third
fundamental form of $M_{+}$. That is,
$$
q_a(X):=<Q(X,X,X),n_a>/3,
$$
where
\begin{eqnarray}\label{eq0}
\aligned
Q(X,Y,Z)&:=(D^{\perp}_XS)(Y,Z)\\
&=D^{\perp}_X(S(Y,Z))-S(\nabla_XY,Z)-S(Y,\nabla_XZ),
\endaligned
\end{eqnarray}
for $0\leq a\leq m_1$, where $D^{\perp}$ is the normal connection and $\nabla$
is the Riemannian connection of $M_{+}$.
These homogeneous polynomials belong to the polynomial ring ${\mathcal P}$ in
$m_1+2m_2$ variables corresponding to the dimension of $M_{+}$.

To apply Proposition~\ref{prop1}, suppose we have established that $p_0,p_1,\cdots,p_{m_1}$ form a regular
sequence in ${\mathcal P}$.

Let us recall an identity of Ozeki and Takeuchi~\cite[I, p 530]{OT}. Namely,
\begin{equation}\label{eq1}
 p_0q_0+p_1q_1+\cdots+p_{m_1}q_{m_1}=0.
\end{equation}
We can interpret equation~\eqref{eq1} in the spirit of Proposition~\ref{prop1} above.
Namely, consider the homogeneous polynomial
$$
F(t_0,t_1,\cdots,t_{m_1})= q_0t_0+q_1t_1+\cdots+q_{m_1}t_{m_1}
$$
over the ring ${\mathcal P}$. Since by assumption $p_0,p_1,\cdots,p_{m_1}$
form a regular sequence
with $F(p_0,p_1,\cdots,p_{m_1})=0$ by~\eqref{eq1},
it follows from Proposition~\ref{prop1} that $q_0,q_1,\cdots,q_{m_1}$ belong
to the ideal $(p_0,p_1,\cdots,p_{m_1})$. That is, we have

\begin{equation}\label{eq2}
q_a=\sum_{b=0}^{m_1}r_{ab}p_b,
\end{equation}
for $0\leq a\leq m_1$, where $r_{ab}$ are homogeneous polynomials of degree 1. 
Substituting~\eqref{eq2} into~\eqref{eq1}, we obtain

\begin{equation}\label{eq3}
\sum_{ab}(r_{ab}+r_{ba})p_ap_b=0.
\end{equation}
Considering the homogeneous polynomial
$$
F(t_0,t_1,\cdots,t_{m_1})=\sum_{ab}(r_{ab}+r_{ba})t_at_b
$$
over the ring ${\mathcal P}$ and observing that $F(p_0,p_1,\cdots,p_{m_1})=0$
by~\eqref{eq3}, Proposition~\ref{prop1} implies $r_{ab}+r_{ba}$ belong
to the ideal $(p_0,p_1,\cdots,p_{m_1})$. However,
this forces $r_{ab}+r_{ba}=0$ since it is homogeneous of degree 1 whereas
$p_0,p_1,\cdots,p_{m_1}$ are homogeneous of degree 2. We thus conclude that
in~\eqref{eq2} we have

\begin{equation}\label{eq4}
r_{ab}=-r_{ba}.
\end{equation}

Now we introduce the Euclidean coordinates of the eigenspaces, with eigenvalues
1, -1, 0, of the shape operator $S_{n_0}$ to be
$u_\alpha,1\leq\alpha\leq m_2,$ and $v_\mu,1\leq\mu\leq m_2,$
and $w_p,1\leq p\leq m_1$, respectively. Set

\begin{equation}\label{eq5}
r_{ab}:=\sum_{\alpha} T_{ab}^{\alpha}u_\alpha+\sum_{\mu}T_{ab}^{\mu}v_\mu+\sum_{p}T_{ab}^pw_p.
\end{equation}
By~\cite[p18]{CCJ}, 

\begin{eqnarray}\label{eqnarray}\label{eq6}
\aligned
2p_0&=\sum_{\alpha}(u_\alpha)^2-\sum_{\mu}(v_\mu)^2,\\
p_a&=\sum_{\alpha\mu}S^a_{\alpha\mu}u_\alpha v_\mu+\sum_{\alpha p}S^a_{\alpha p}u_\alpha w_p+\sum_{\mu p}S^a_{\mu p}v_{\mu}w_p,
\endaligned
\end{eqnarray}
for $1\leq a\leq m_1$, where we set
$$
S^a_{\alpha\mu}:=<S(X_\alpha,Y_\mu),n_a>,
$$
etc.,
with $X_\alpha,Y_\mu,$ and $Z_p$ the orthonormal bases for the coordinates $u_\alpha,v_\mu,$
and $w_p$, respectively.
Note that our $p_a$ are different from those in equation (6.6) of~\cite{CCJ}, 
which are truncated version of ours. We claim that

\begin{equation}\label{eq7}
T^\alpha_{a0}=T^\mu_{a0}=0,
\end{equation}
for $1\leq a\leq m_1$. To this end, we calculate $q_a$ in two ways. On the one hand,
substituting~\eqref{eq5} and~\eqref{eq6} into~\eqref{eq2}, we see that $q_a$
has the term 
$$
(\sum_\alpha T^\alpha_{a0}u_\alpha)(\sum_\beta (u_\beta)^2)/2,
$$
so that the coefficient of $(u_\alpha)^3$ in $q_a$, denoted by $q_a^{\alpha\alpha\alpha}$, is
$$
q_a^{\alpha\alpha\alpha}=T^\alpha_{a0}/2.
$$
On the other hand, when calculating $D^{\perp}$, we can pick a normal frame
so that the normal connection form is zero at any fixed point. Then
by~\eqref{eq0} and the fact $S^a_{\alpha\beta}=S^a(X_\alpha,X_\beta)=0$, we calculate to see

\begin{eqnarray}\label{cal}
\aligned
3T^\alpha_{a0}/2&=3q_a^{\alpha\alpha\alpha}=<Q(X_\alpha,X_\alpha,X_\alpha),n_a>\\
&=dS^a_{\alpha\alpha}(X_\alpha)-\sum_{t}\theta^t_{\alpha}(X_\alpha)S^a_{t\alpha}-\sum_{t}\theta^t_{\alpha}(X_\alpha)S^a_{\alpha t}\\
&=-2\sum_t\theta^t_\alpha(X_\alpha)S^a_{t\alpha}=-2\sum_p\theta^p_\alpha(X_\alpha)S^a_{p\alpha}-2\sum_\mu\theta^\mu_\alpha(X_\alpha)S^a_{\alpha\mu}\\
&=0,
\endaligned
\end{eqnarray}
where $\theta^i_j$ is the Riemannian connection forms of $M_{+}$ and
the last equality is by equation (4.18) of~\cite{CCJ}. Likewise, $T^\mu_{a0}=0$. Hence,~\eqref{eq7} is
proven. The skew-symmetry of $r_{ab}$ in $a,b$ then yields
\begin{equation}\label{eq8}
T^\alpha_{0a}=T^\mu_{0a}=0.
\end{equation}

Next, let us calculate $q_0$ in two ways. On the one hand,
we expand $q_0$
by~\eqref{eq2},~\eqref{eq5},~\eqref{eq6} and~\eqref{eq8}, keeping in mind
that $q_0$ is homogenous of degree 1 in $u_\alpha,v_\mu$ and $w_p$,
by~\cite[I, p 537]{OT}, to obtain that the coefficient of the $u_\alpha v_\mu w_p$-term
of $q_0$, denoted by $q_0^{\alpha\mu p}$, is
$$
q_0^{\alpha\mu p}=\sum_{b\geq 1}T^p_{0b}S^b_{\alpha\mu}
=2\sum_{b\geq 1}T^p_{0b}F^\mu_{\alpha b},
$$                        
where $S^{b}_{\alpha\mu}=2F^{\mu}_{\alpha a}$ is employed in
equation (6.4) of~\cite{CCJ}.
On the other hand, by~\eqref{eq0} and $p_0$ in~\eqref{eq6}, a similar calculation
as in~\eqref{cal} yields
\begin{eqnarray}\nonumber
\aligned
q_0^{\alpha\mu p}&=2<Q(Y_\mu,Z_p,X_\alpha),n_0>\\
&=2(dS^0_{p\alpha}(Y_\mu)-\sum_{t}\theta^{t}_{p}(Y_\mu)S^0_{t\alpha}-\sum_{t}\theta^t_{\alpha}(Y_\mu)S^0_{p t})\\
&=-2\sum_{t}\theta^{t}_{p}(Y_\mu)S^0_{t\alpha}=-2\theta^\alpha_p(Y_\mu)\\
&=4F^\mu_{\alpha p},
\endaligned
\end{eqnarray}
where the last equality follows by equation (4.18) of~\cite{CCJ}. In conclusion, we derive
\begin{equation}\label{I}
F^\mu_{\alpha p}=\sum_b f_{pb} F^\mu_{\alpha b},
\end{equation}
where
\begin{equation}\label{II}
f_{pb}=T^p_{0b}/2,
\end{equation}
which is exactly equation (6.13) of~\cite{CCJ}.
Therefore, we may assume, as in Proposition 11~\cite[p 19]{CCJ}, that
\begin{equation}\label{anti-sym1}
F^\mu_{\alpha\, a+m_1}= F^\mu_{\alpha a},
\end{equation}
where we now agree that $p$ is indexed between $m_1+1$ and $2m_1$. With this
index choice, we see
$$
f_{a+m_1\; b}=\delta_{ab}
$$
by~\eqref{I} and~\eqref{anti-sym1}. That is,
by~\eqref{eq5},~\eqref{eq8} and~\eqref{II},
$$
r_{0b}=2\sum_{a}\delta^{a}_{b}w_{a+m_1}=2w_{b+m_1},
$$
and, with the Einstein summation convention,
\begin{equation}\nonumber
\aligned
&q_0=r_{0b}p_b\\
&=2(\delta^{a}_{b}w_{a+m_1})
(S^b_{\alpha\mu}u_\alpha v_\mu+S^b_{\alpha\; c+m_1}u_\alpha w_{c+m_1}+S^b_{\mu\; c+m_1}v_{\mu}w_{c+m_1}).
\endaligned
\end{equation}
Hence, we have
$$
\sum_{a b c\alpha}(\delta^{a}_{b}w_{a+m_1})(S^b_{\alpha\; c+m_1}u_\alpha w_{c+m_1})=0,
$$
or equivalently, noting that $S^a_{\alpha\; c+m_1}=-F^{\alpha}_{c+m_1\; a}$
by~\cite[p 18]{CCJ},  
$$
\sum_{ac}F^{\alpha}_{c+m_1\; a}w_{c+m_1}w_{a+m_1}=0.
$$
In other words, we have
\begin{equation}\label{anti-sym2}
F^{\alpha}_{c+m_1\; a}=-F^{\alpha}_{a+m_1\; c}.
\end{equation}
Likewise, We have
\begin{equation}\label{anti-sym3}
F^{\mu}_{c+m_1\; a}=-F^{\mu}_{a+m_1\; c}.
\end{equation}

\begin{lemma}\label{span}
If $m_1<m_2$, then the vectors
$$
(F^\mu_{\alpha 1},F^\mu_{\alpha 2},\cdots,F^\mu_{\alpha m_1}),
$$
for $1\leq\alpha,\mu\leq m_2,$ span ${\mathbb R}^{m_1}$.
\end{lemma}

\begin{proof} This is Proposition 7 of~\cite[p 18]{CCJ}.
\end{proof}

\begin{lemma}\label{iso}
Let $m_1<m_2$. If~\eqref{anti-sym1},~\eqref{anti-sym2},~\eqref{anti-sym3} hold,
then the hypersurface is of OT-FKM type.
\end{lemma}

\begin{proof} It suffices to show that equation (8.4) of~\cite[p 28]{CCJ} holds. Then
Theorem 24~\cite[p 36]{CCJ} (or see~\cite{Ch} for a conceptual proof of it) will establish the conclusion.

Since all the formulae are given in the proof of Proposition 19
of~\cite{CCJ}, we shall be brief and shall follow all the notations there.

Employing equations (8.5) through (8.9) of~\cite{CCJ}, one establishes,
by~\eqref{anti-sym1} through~\eqref{anti-sym3}
with the Einstein summation convention, that
$$
F^\mu_{\alpha b}(-\theta^b_a+\theta^{b+m_1}_{a+m_1})=(F^\mu_{\alpha\; a\;c}-F^\mu_{\alpha\; a+m_1\;c})\theta^c+
(F^\mu_{\alpha\; a\;c+m_1}-F^\mu_{\alpha\; a+m_1\;c+m_1})\theta^{c+m_1}.
$$
Then (8.21) through (8.24) of~\cite{CCJ} imply that
$$
F^\mu_{\alpha\; a\;c}-F^\mu_{\alpha\; a+m_1\;c}=
F^\mu_{\alpha\; a\;c+m_1}-F^\mu_{\alpha\; a+m_1\;c+m_1},
$$
so that $\theta^b_a-\theta^{b+m_1}_{a+m_1}$ is only spanned by
$\theta^c+\theta^{c+m_1},1\leq c\leq m_1,$ by Lemma~\ref{span} above. This proves
equation (8.4) of~\cite{CCJ}.
\end{proof}

We summarize what we have done so far in the following satisfying proposition.

\begin{proposition}\label{prop2} Let $m_1<m_2$. If $p_0,p_1,\cdots,p_{m_1}$ form a regular
sequence,
then the isoparametric hypersurface is of OT-FKM type.
\end{proposition}

\section{The dimension estimate when $m_2\geq 2m_1-1$}
We now show that if $m_2\geq 2m_1-1$, then the assumption in
Proposition~\ref{prop2} holds so that the isoparametric hypersurface is of
OT-FKM type. Henceforth, all homogeneous polynomials are regarded as being over the complex numbers.

We agree that ${\mathbb C}^{2m_2+m_1}$ consists of
points $(u,v,w)$
with coordinates $u_\alpha,v_\mu$ and $w_p$, where $1\leq\alpha,\mu\leq m_2$
and $1\leq p\leq m_1$. For $0\leq k\leq m_1$, let
$$
W_k:=\{(u,v,w)\in{\mathbb C}^{2m_2+m_1}:p_0(u,v,w)=\cdots=p_k(u,v,w)=0\}.
$$
We want to estimate the dimension of the subvariety $U_k$ of ${\mathbb C}^{2m_2+m_1}$, where
$$
U_k:=\{(u,v,w)\in{\mathbb C}^{2m_2+m_1}
:\text{rank of the Jacobian of}\; p_0,\cdots,p_k<k+1\}.
$$
Similar to~\cite[p 68]{CCJ}, $p_0,\cdots,p_{k}$ give rise to a linear system of
cones ${\mathcal C}_\lambda$ defined by
$$
c_0p_0+\cdots+c_{k}p_{k}=0
$$
with $\lambda=[c_0:\cdots:c_{k}]\in{\mathbb C}P^{k}$. The
singular subvariety of ${\mathcal C}_\lambda$ is
$$
{\mathscr S}_\lambda:=\{(u,v,w)\in{\mathbb C}^{2m_2+m_1}:
(c_0S_{n_0}+\cdots+c_kS_{n_k})\cdot (u,v,w)^{tr}=0\},
$$
where $<S_{n_i}(X),Y>=<S(X,Y),n_i>$ is the shape operator of the focal manifold
$M_{+}$ in the normal direction $n_i$; we have 
\begin{equation}\label{union}
U_k=\cup_\lambda{\mathscr S}_\lambda.
\end{equation}
By corollary~\ref{cor1}, we wish to establish 
\begin{equation}\label{prime}
\dim(W_k\cap U_k)\leq\dim(W_k)-2
\end{equation}
for $k\leq m_1-1$ to verify that $p_0,p_1,\cdots,p_{m_1}$ form a regular sequence.

We first estimate the dimension of ${\mathscr S}_\lambda$. We break it into two cases.
If $c_0,\cdots,c_{k}$ are either all real or all purely imaginary, then
$$
\dim({\mathscr S}_\lambda)=m_1,
$$
since $c_0S_{n_0}+\cdots+c_kS_{n_k}=cS_n$ for some unit normal vector $n$ and
some nonzero real or purely imaginary constant $c$, and we know that the null space
of $S_n$ is of dimension $m_1$. On the other hand, if $c_0,\cdots,c_k$ are not all
real and not all purely imaginary, then similar to~\cite[p 69]{CCJ}, after a normal
basis change,
we can assume that ${\mathscr S}_\lambda$ consists of elements $(u,v,w)$
of the form $(S_{n^*_1}-\tau S_{n^*_0})\cdot(u,v,w)^{tr}=0$ for some nonzero complex number
$\tau$, relative to a new orthonormal normal basis $n^*_0,n^*_1,\cdots,n^*_{k}$
in the linear span of $n_0,n_1,\cdots,n_k$. That is, in matrix form,

\begin{equation}\label{matrix}
\begin{pmatrix}0&A&B\\A^{tr}&0&C\\B^{tr}&C^{tr}&0\end{pmatrix}\begin{pmatrix}x\\y\\z\end{pmatrix}
=\tau\begin{pmatrix}I&0&0\\0&-I&0\\0&0&0\end{pmatrix}\begin{pmatrix}x\\y\\z\end{pmatrix},
\end{equation}
where $x,y$ and $z$ are (complex) eigenvectors of (real) $S_{n^*_0}$ with eigenvalues $1,-1$ and $0$, respectively. Lemma 49~\cite[p 64]{CCJ} ensures that we can assume
$$
B=C=\begin{pmatrix}0&0\\0&\sigma\end{pmatrix},
$$
where $\sigma$ is a nonsingular diagonal matrix of size $r$-by-$r$ with $r$ the rank of $B$,
and $A$ is of the form
\begin{equation}\label{A}
A=\begin{pmatrix}I&0\\0&\Delta\end{pmatrix},
\end{equation}
where $\Delta=\text{diag}(\Delta_1,\Delta_2,\Delta_3,\cdots)$
is of size $r$-by-$r$, in which $\Delta_1=0$
and $\Delta_i,i\geq 2,$ are nonzero skew-symmetric matrices expressed in the block form
$\Delta_i=\text{diag}(\Theta_i,\Theta_i,\Theta_i,\cdots)$ with $\Theta_i$ a 2-by-2 matrix
of the form
$$
\begin{pmatrix}0&f_i\\-f_i&0\end{pmatrix}
$$
for some $0<f_i<1$. We decompose $x,y,z$ into $x=(x_1,x_2),y=(y_1,y_2),z=(z_1,z_2)$
with $x_2,y_2,z_2\in{\mathbb C}^r$ (by abuse of notation we do not distinguish
column vectors from row vectors).
Then~\eqref{matrix} comes down to

\begin{eqnarray}\label{estimate}
\aligned
x_1=-\tau y_1,&\qquad y_1=\tau x_1,\\
-\Delta x_2+\sigma z_2=-\tau y_2,&\qquad \Delta y_2+\sigma z_2=\tau x_2,\\
\Delta(x_2&+y_2)=0.
\endaligned
\end{eqnarray}
It follows from the first set of equations of~\eqref{estimate} that either $x_1=y_1=0$, or both are nonzero
with $\tau=\pm\sqrt{-1}$. In both cases, by the second set of equations
of~\eqref{estimate}, we have
$$
(\Delta^2-\tau^2I)x_2=(\Delta-\tau I)\sigma z_2,\qquad (\Delta^2-\tau^2I)y_2=-(\Delta-\tau I)\sigma z_2,
$$
which together with the third equation of~\eqref{estimate} imply that 
$x_2=-y_2$, and so $z_2$ can be solved in terms of $x_2$ by the second set of
equations of~\eqref{estimate}. (Note that conversely $x_2=-y_2$
can be solved in terms of $z_2$ when $\tau\neq\pm f_i\sqrt{-1}$ for all $i$, so that $z$
can be chosen to be a free variable
in this case.) So, either $x_1=y_1=0$, in which case
$$
\dim({\mathscr S}_\lambda)=m_1,
$$
or both $x_1$ and $y_1$ are nonzero,
in which case $y_1=\pm\sqrt{-1}x_1$ and so
\begin{equation}\label{EST}
\dim({\mathscr S}_\lambda)=m_1+m_2-r.
\end{equation}
Since eventually we must estimate the dimension of $W_k\cap U_k$, let us cut
${\mathscr S}_\lambda$ by
$$
0=p^*_0=\sum_{\alpha}(x_\alpha)^2-\sum_{\mu}(y_\mu)^2.
$$
\noindent Case 1. $x_1$ and $y_1$ are both nonzero. This is the case of
nongeneric $\lambda\in{\mathbb C}P^k$. We substitute $y_1=\pm \sqrt{-1}x_1$ and $x_2$
and $y_2$ in terms of $z_2$ into $p^*_0=0$ to
deduce
$$
0=p^*_0=(x_1)^2+\cdots+(x_{m_2-r})^2+z\; \text{terms};
$$
hence $p^*_0=0$ cuts ${\mathscr S}_\lambda$ to reduce the dimension by 1, i.e., by~\eqref{EST},
\begin{equation}\label{sub}
\dim(W_{k}\cap{\mathscr S}_\lambda)\leq (m_1+m_2-r)-1\leq m_1+m_2-1,
\end{equation}
noting that $W_k$ is also cut out by $p^*_0,p^*_1,\cdots,p^*_k$. Meanwhile,
similar to~\cite[p 71]{CCJ}, only a subvariety of $\lambda$ of
dimension $k-1$ in ${\mathbb C}P^k$ assumes $\tau=\pm\sqrt{-1}$. 
Therefore, by~\eqref{sub}, an irreducible component ${\mathcal W}$
of $W_k\cap \cup_\lambda {\mathscr S}_\lambda$ over nongeneric $\lambda$
will satisfy
\begin{equation}\nonumber
\dim({\mathcal W})\leq\dim(W_{k}\cap{\mathscr S}_\lambda)+k-1\leq
m_1+m_2+k-2.
\end{equation}
\noindent Case 2. $x_1=y_1=0$. This is the case of generic $\lambda$, where
$\dim({\mathscr S}_\lambda)=m_1$, so that an irreducible component ${\mathcal V}$
of $W_k\cap \cup_\lambda{\mathscr S}_\lambda$ over generic $\lambda$ will satisfy
$$
\dim({\mathcal V})\leq m_1+k\leq m_1+m_2+k-2
$$
as we may assume $m_2\geq 2$. (The case $m_1=m_2=1$ is straightforward~\cite[p 61]{CCJ}.)

Putting these two cases together, we conclude
\begin{equation}\label{crucial}
\dim(W_k\cap U_k)\leq m_1+m_2+k-2.
\end{equation}
On the other hand, since $W_k$ is cut out by $k+1$ equations, we have
\begin{equation}\label{lower-bound}
\dim(W_k)\geq m_1+2m_2-k-1.
\end{equation}
Therefore,
$$
\dim(W_k\cap U_k)\leq \dim(W_k)-2
$$
when $k\leq m_1-1$, taking $m_2\geq 2m_1-1$ into account.

In summary, we have established~\eqref{prime} for $k\leq m_1-1$,
so that the ideal $(p_0,p_1,\cdots,p_k)$ is prime when $k\leq m_1-1$.
Corollary~\ref{cor1} 
then imply that $p_0,p_1,\cdots,p_{m_1}$ form a regular sequence.
It follows by Proposition~\ref{prop2} 
that the isoparametric hypersurface is of OT-FKM type.

\section{The case $(m_1,m_2)=(3,4)$}
Recall that a point of $M_{+}$ is of Condition A if
the shape operators $S_n$ share the same kernel for all unit normal $n$.
Conditions A and B were introduced and explored by
Ozeki and Takeuchi~\cite[I, p 541]{OT} in their construction of inhomogeneous isoparametric
hypersurfaces in spheres. It was then established by Dorfmeister and Neher~\cite{DN1}
(see also~\cite{Chi}) that Condition A alone implies that the isoparametric hypersurface
is of OT-FKM type.

\begin{theorem} Let $(m_1,m_2)=(3,4)$. Then there exist points 
of Condition A on $M_{+}$;
the isoparametric hypersurface is then of OT-FKM type.
\end{theorem}
\begin{proof}
We follow the notations in the previous section. Suppose $M_{+}$ is free of points of Condition A everywhere. Then one of the three pairs of
matrices $(B_1,C_1),(B_2,C_2)$ and $(B_3,C_3)$ of the shape operators $S_{n_1},S_{n_2}$
and $S_{n_3}$, similar to the one given in~\eqref{matrix}, must be nonzero. By replacing $(B_2,C_2)$ by $(B_3,C_3)$, we may assume
one of $(B_1,C_1)$ and $(B_2,C_2)$ is nonzero in the neighborhood of a given point.
For $k=m_1-1=2$, observe that in~\eqref{sub} if $r>0$ holds, then in fact
\begin{equation}\label{EQ1}
\dim(W_{2}\cap{\mathscr S}_\lambda)\leq m_1+m_2-2.
\end{equation}
If $r=0$, then $B^*_1=C^*_1=0$ and $A^*=I$ in~\eqref{matrix} for $S_{n^*_1}$.
It follows that $p^*_0=0$ and $p^*_1=0$ cut ${\mathscr S}_\lambda$ in the variety
\begin{equation}\label{free}
\{(x,\pm \sqrt{-1}x,z):\sum_\alpha (x_\alpha)^2=0\}.
\end{equation}
$(B^*_2,C^*_2)$ of $S_{n^*_2}$ must be nonzero now. 
Since $z$ is a free variable in~\eqref{free},
$p^*_2=0$ will have nontrivial $z$-terms
\begin{eqnarray}\nonumber
\aligned
0=p^*_2&=\sum_{\alpha p}S_{\alpha p}x_\alpha z_p+\sum_{\mu p}T_{\mu p}y_\mu z_p+x_\alpha y_\mu\;\text{terms}\\
&=\sum_{\alpha p}(S_{\alpha p}\pm\sqrt{-1}T_{\alpha p})x_\alpha z_p+x_\alpha x_\mu\;\text{terms},
\endaligned
\end{eqnarray}
taking $y=\pm\sqrt{-1}x$ into account, where $S_{\alpha p}:=<S(X^*_\alpha,Z^*_p),n^*_2>$
and $T_{\mu p}:=<S(Y^*_\mu,Z^*_p),n^*_2>$ are (real) entries of $B^*_2$ and $C^*_2$,
respectively, and $X^*_{\alpha},1\leq\alpha\leq m_2$,
$Y^*_{\mu},1\leq\mu\leq m_2$ and $Z^*_{p},1\leq p\leq m_1$, are orthonormal
eigenvectors for the eigenspaces of $S_{n^*_0}$ with eigenvalues $1,-1,$ and $0$,
respectively;
hence the dimension
of ${\mathscr S}_\lambda$ will be cut down by 2 by $p^*_0,p^*_1,p^*_2=0$, so that again
\begin{equation}\label{EQ2}
\dim(W_{2}\cap{\mathscr S}_\lambda)\leq m_1+m_2-2,
\end{equation} 
noting that $p^*_0,p^*_1,p^*_2=0$ also cut out $W_2$.

In conclusion, for $k=2$,
\begin{equation}\label{UB}
\dim(W_{k}\cap U_{k})\leq\dim(W_{k}\cap{\mathscr S}_\lambda)+k-1\leq
m_1+m_2+k-3=6,
\end{equation}
while by~\eqref{lower-bound},
$$
\dim(W_k)\geq 8,
$$
so that
$$
\dim(W_{2}\cap U_{2})\leq\dim(W_{2})-2.
$$
(Again, it suffices to consider the nongeneric case for the dimension estimate
since the generic case only contributes dimension at most $m_1+k=5$,
similar to what is detailed in the previous section.)

Meanwhile,~\eqref{crucial} and~\eqref{lower-bound} for $k\leq 1$ imply
$$
\dim(W_k\cap U_k)\leq \dim(W_k)-2.
$$
Therefore, it follows again by~\eqref{prime} that $(p_0),(p_0,p_1)$
and $(p_0,p_1,p_2)$ are all prime ideals and Corollary~\ref{cor1}
asserts that $p_0,p_1,p_2,p_3$ form a regular sequence, so that
Proposition~\ref{prop2} 
establishes that the isoparametric hypersurface is of the type
constructed by Ozeki and Takeuchi~\cite[I]{OT},
which has points of Condition A on $M_{+}$.
This contradiction to the assumption
made at the outset shows that indeed $M_{+}$ has points of Condition A. 
\end{proof}

\section{Concluding remarks} Knowing that $M_{+}$ has points of Condition A for
an isoparametric hypersurface with multiplicities $(7,8)$ of OT-FKM type~\cite{FKM}, one is tempted to
apply the upper bound estimate
$m_1+m_2+k-3$ in~\eqref{UB} and the lower bound estimate $m_1+2m_2-k-1$ in~\eqref{lower-bound}
to the case, with $k=6$,
to assert~\eqref{prime}. This encounters apparent difficulties.

In fact, the upper bound cannot be improved to $m_1+m_2+k-4$ without 
understanding further properties pertaining to an isoparametric hypersurface
with one of the three remaining multiplicity pairs,
since otherwise $(m_1,m_2)=(4,5)$ for $k\leq 3$ would satisfy
$$\dim(W_k\cap U_k)\leq \dim(W_k)-2,
$$
so that the same arguments as above would imply that the isoparametric
hypersurface would be of OT-FKM type, which is not the case.

A detailed understanding of the number $r$ in~\eqref{sub} for
the multiplicity pair $(7,8)$ would seem important.


\end{document}